\documentclass[11pt]{article}
\usepackage{amssymb,latexsym}
\usepackage{graphicx}
\usepackage{fancyhdr}
\usepackage{mathdots}
\usepackage{cite}
\usepackage[english]{babel}
\usepackage{nicefrac}
\usepackage{bbm}
\usepackage{amssymb}
\usepackage{tabularx}
\usepackage{epsfig}
\usepackage{amsfonts}
\usepackage{amscd}
\usepackage{amsmath,amsthm}
\usepackage{lmodern}
\RequirePackage{wasysym} \RequirePackage{setspace}
\usepackage{enumerate}
\usepackage{mathrsfs}
\usepackage[linktocpage, colorlinks=true,citecolor=blue,linkcolor=blue,urlcolor=blue]{hyperref}

\newcommand{\arxiv}[1]{\href{http://arxiv.org/abs/#1}{\texttt{arXiv:#1}}}

\numberwithin{equation}{section}
\theoremstyle{plain}
\newtheorem*{thm}{Theorem}
\newtheorem{theorem}{Theorem}[section]
\newtheorem{lemma}[theorem]{Lemma}
\newtheorem{corollary}[theorem]{Corollary}
\newtheorem{proposition}[theorem]{Proposition}
\newtheorem{fact}[theorem]{Fact}

\theoremstyle{definition}
\newtheorem{definition}[theorem]{Definition}
\newtheorem{example}[theorem]{Example}
\newtheorem{conjecture}[theorem]{Conjecture}

\theoremstyle{remark}
\newtheorem{remark}[theorem]{Remark}

\newcommand{\be}{\begin{equation}}
\newcommand{\ee}{\end{equation}}
\newcommand{\bea}{\begin{eqnarray}}
\newcommand{\eea}{\end{eqnarray}}
\newcommand{\bfa}{\begin{fact}}
\newcommand{\efa}{\end{fact}}
\newcommand{\bin}{\begin{inequality}}
\newcommand{\ein}{\end{inequality}}
\newcommand{\pnom}[3]{{#1 \choose #2}_{\hskip -3pt \mathrm{#3}}}
\newcommand{\enom}[3]{{#1 \choose #2}_{\hskip -3pt #3}}

\def \a {{\mathbf a}}
\def\b {{\mathbf b}}

\def\T {{\mathsf{T}_{\hskip-1pt m}}}

\title{\bf  Polynomial Triangles Revisited}

\author{Nour-Eddine Fahssi\\
\small Lab of Mathematics, Cryptography and Mechanics,\\[-0.8ex]
\small FSTM, University of Hassan II - Mohammedia, BP 146, \\[-0.8ex]
\small Mohammedia,  Morocco.\\
\small and \\
\small Lab of High Energy Physics, Modeling and Simulation,\\[-0.8ex]
 \small Faculty of Science, University of Mohammed V-Agdal, \\[-0.8ex]
\small  Rabat, Morocco. \\
\small \href{mailto:fahssi@fstm.ac.ma}{\tt fahssi@fstm.ac.ma}}

\date{}
\begin{document}
\maketitle


\begin{abstract}
 A polynomial triangle is an array whose inputs are the coefficients in integral powers of a polynomial. Although polynomial coefficients have appeared in several works, there is no systematic treatise on this topic. In this paper we plan to fill this gap.  We describe some aspects of these arrays, which generalize similar properties of the binomial coefficients. Some combinatorial models enumerated by polynomial coefficients, including lattice paths model, spin chain model and scores in a drawing game, are introduced. Several known binomial identities are then extended. In addition, we recursively calculate  generating functions of column sequences.  Interesting corollaries follow from these recurrence relations such as new formulae for the Fibonacci numbers and Hermite polynomials  in terms of trinomial coefficients.  Finally, we study some properties of the entropy density function which characterizes polynomial triangles in the thermodynamical limit.

\end{abstract}
\tableofcontents

\section{ Introduction}
The theme of our study is extremely simple. It consists in investigating several aspects of  coefficients in integral  powers of polynomials. These coefficients generate an array, called a polynomial triangle, whose $k$-th row consists of the coefficients of the powers of $t$ in $p(t)^k$, for a given polynomial $p(t)$. The table naturally resembles Pascal's triangle and reduces to it when $p(t)=1+t$.

Historically, this very natural extension of the Pascal triangle may have been first discussed by Abraham De Moivre~\cite{moivre,balasu}  who found that the coefficient of $t^n$ of the polynomial \be \label{sans}\left(1+t+t^2+ \cdots +t^{m}\right)^k, \ee  ($k \geq 0$) arises in the solution of the following problem~\cite[p.389]{hall}:
\begin{center} \begin{minipage}[b]{4.5 in} ``There are $k$ dice with $m+1$ faces marked from 1 to $m+1$; if these
are thrown at random, what is the chance that the sum of the numbers
exhibited shall be equal to $n$?''\end{minipage} \end{center} A few decades later, Leonhard Euler~\cite{Euler1,Euler2} published an analytical study of the coefficients of the polynomial~\eqref{sans}. The elementary
properties of the array generated by these coefficients closely mimic
those of binomial ones. For example, each entry in the
body of the (centered) triangle is the sum of the $m$ entries
above it, extending a well-known property of Pascal's triangle.
This array, denoted in this paper by $\T$, is termed the
\emph{extended Pascal triangle}~\cite{Boll3}, or Pascal-T
triangle~\cite{Turner}, or Pascal-De Moivre triangle~\cite{larry}. In 1937, it was reintroduced by A.~Tremblay~\cite{tremblay} and in 1942 by P.~Montel~\cite{montel},
and explicitly discussed by John Freund in 1956~\cite{Freund}, as arising in the
solution of a restricted occupancy problem. In fact, the
coefficient of $t^n$ in \eqref{sans}, denoted by George Andrews $\enom{k}{n}{\, m}$~\cite{Andrews1}, is the number of distinct ways in which $n$ unlabeled objects can be distributed in $k$ labeled boxes allowing at most
$m$ objects to fall in each box (see also the Riordan monograph~\cite[p.104]{riordan1}). In statistical
physics, the Freund restricted occupancy model is known as the Gentile intermediate statistics, called after Giovanni Gentile Jr~\cite{gent1,gent2}. This model interpolates between Fermi-Dirac statistics (binomial case :  $m=1$) and Bose-Einstein statistics ($m=\infty$).  It will be referred here to as \emph{Gentile-Freund statistics}
(GFS).

The arrays $\T$ have been extensively used  in reliability and probability
studies \cite{balasu, daren,sen,makri}. Several results about the extended binomial coefficients, specially the trinomial
ones ($m=2$), are known~\cite{Boll3,Bankier,Boll1,Boll2,Boll4,Bonda,Fielder}. Some
generalizations have been discussed as well. For instance, Ollerton and Shannon~\cite{ollerton} have investigated
various properties and applications of a generalization of the
triangle $\T$ by extending the Freund occupancy model.
We underline in passing that the entries of $\T$ have been
$q$-generalized by George Andrews and Rodney J. Baxter~\cite{Andrews3} for $m=2$ to solve the hard hexagon model
in statistical mechanics and later by Warnaar~\cite{warnaar}
for arbitrary $m$. This $q$-analog proved to have a deep
connection with the Rogers-Ramanujan identities and the theory of
partitions~\cite{Andrews1,Andrews2}.

The Pascal-De Moivre triangles $\mathsf{T}_{\hskip-1pt 2}$ (which begins as shown in Figure~\ref{fig}),
$\mathsf{T}_{\hskip-2pt 3}$ and $\mathsf{T}_{\hskip-2pt 4}$ are
recorded in Sloane's \emph{Online Encyclopedia of Integer Sequences}~\cite{sloane} as A027907, A008287 and
A035343 respectively.

\begin{figure}
\footnotesize{\[ \begin{array}{ccccccccccccc}
  &  &  &  &  &  &1  &  &  &  &  & & \\
  &  &  &  &  & 1 & 1 & 1 &  &  &  &  &\\
  &  &  &  & 1 & 2 & 3 & 2 & 1 &  &  & & \\
  &  &  & 1 & 3 & 6 & 7 & 6 & 3 & 1 &  &  &\\
 &   & 1 & 4 & 10 & 16 & 19 & 16 & 10 & 4 & 1 & & \\
  & 1 & 5 & 15 & 30 & 45 & 51 & 45 & 30 & 15 & 5 & 1&\\
 \iddots &\vdots & \vdots & \vdots & \vdots & \vdots & \vdots & \vdots & \vdots & \vdots & \vdots & \vdots & \ddots
 \end{array}\]}
 \\
  \caption{The centered triangle $\mathsf{T}_{\hskip-1pt 2}$}\label{fig}
\end{figure}

\subsection{Preliminaries} Let us fix our terminology and notations.
\begin{definition} \label{def1}
let $\a$ be a sequence of $m+1$ numbers $(a_0,a_1,\ldots
,a_{m})$ and let $p_a(t)=\sum_{i=0}^m a_i t^i$ be its generating polynomial. The \emph{polynomial coefficients}
associated with the vector $\a$ are defined by\footnote{We make
use of the conventional notation for coefficients of entire series
: $\left[t^n\right]\sum_i a_i t^i~:= a_n$.} \be \label{pnom}  \pnom{k}{n}{\a} \stackrel{\mbox{\tiny def}}=  \left\{
                                                     \begin{array}{ll}
                                                       [t^n]\left(p_\a(t)\right)^k, & \hbox{if} \quad 0 \leq n \leq mk \\
                                                       0, & \hbox{if} \quad n <0 \; \; \hbox{or} \; \; n > mk
                                                     \end{array}
                                                   \right.\ee where we
have used a vector-indexed binomial symbol mimicking the notation
of Andrews. When $a_i=1$ for all $i$, the binomial symbol will be
simply indexed by $m$. The array of polynomial coefficients will
be called \textit{$\a$-polynomial triangle} or \textit{$\a$-triangle} for short, and denoted by $\mathsf{T}(\a)$.
Rows are indexed by $k$ and columns by $n$. The polynomial triangle $\mathsf{T}(\a)$ will be
called \emph{arithmetical} or {\em combinatorial} if the coefficients $a_i$  are
integers.
\end{definition}
The term ``polynomial coefficients'' is inspired by the designation of Louis Comtet ~\cite[p.78]{Comt} for $\T$. Though apparently shallow, the map $\a \mapsto \mathsf{T}(\a)$ defined by~\eqref{pnom} has quite nontrivial properties which relate to ``deeper'' mathematics. We stress that definition~\ref{def1} appears in~\cite{kallos} as a
consequence of ``a generalization of Pascal's triangle using powers of base
numbers''. We quote also in this context the work of Noe~\cite{noe} who studied in some
detail the central coefficient in $(1+bt+ct^2)^n$ for integers
$b,c$.

By application of the multinomial formula, we see that the polynomial coefficients are homogeneous polynomials of
degree $k$ in the numbers $a_i$: \be \label{pnom1}\pnom{k}{n}{\a}=k! \sum_{\mathbf{k} \, \in \, \mathcal{O}(k,n)}  \frac{\a ^\mathbf{k}}{\mathbf{k}!},\ee
where the sum is over the set $\mathcal{O}(k,n)$ of nonnegative integer vectors $\mathbf{k} = (k_0 ,.., k_m)$  subject to the constraints $k_{0}+ k_{1}+ \cdots+k_{m}=k $ and $k_{1}+2k_{2}+ \cdots +m k_{m}=n$, and where the following concise notations for powers and factorials of a vector are used \[\a^\mathbf{k} = \prod_{i=0}^{m}a_{i}^{k_{i}}, \qquad \mathbf{k}! = \prod_{i=0}^{m}k_i! .\]

We note the following points about the polynomial coefficients. Eq.~\eqref{pnom1} can also be viewed as a ``restricted'' Bell polynomial in the indeterminates $a_i$. In particular cases, \eqref{pnom1} can be expressed in terms of known orthogonal polynomials. For $m=2$, it can be written in terms of the so-called two-variable one-parameter Gegenbauer polynomials defined by $(\alpha-2xt+yt^2)^{-\lambda}=\sum_{n=0}^\infty C_n^{(\lambda)}(x,y;\alpha)t^n$. In particular, if $a_0a_2 >0$, ~\eqref{pnom} is a value of the ordinary (one-variable) ultraspherical polynomials~\cite[p.783]{HMF}:  \be \label{gegen} \pnom{k}{n}{\a}=
a_0^{k-n/2} a_2^{n/2} C_n^{(-k)}\left(-\frac{a_1}{2 \sqrt{a_0
a_2}}\right).\ee

A final point that needs to be addressed is that, with the exception of
the binomial triangle, $\mathsf{T}(\a)$ is a Riordan array only if $p_{\a}(0)=0$, i.e, $a_0=0$, as can be seen from the bivariate generating function :  \[ \sum_{k , n}
\pnom{k}{n}{\a} t^n u^k = \frac{1}{1-u p_{\a}(t)}. \] For basic informations on Riordan arrays, the reader is referred to~\cite{riordan}.

\subsection{Outline and Main results}
Literature on the triangles $\T$ remains quite sparse, and there is no systematic interest in their properties. In this research paper, we plan to fill this gap by presenting a unified approach to the subject. A more ambitious program is to extend all the known properties of the ubiquitous binomial coefficients as proposed by Comtet~\cite{Comt}.

The following items give a sample of our results :
\begin{itemize}
  \item Besides the weighted Gentile-Freund model, we give in Section~\ref{s2} three combinatorial models enumerated by polynomial coefficients : \begin{enumerate}
                                                                                                                                                  \item we show that polynomial coefficients count the number of points a player can score in a game of drawing colored balls (Theorem~\ref{score}). We discuss, by the way, a variant of the old problem of De Moivre and rediscover that polynomial coefficients provide a natural extension of the usual binomial probability distribution based on Bernoulli trials with more than two outcomes.
                                                                                                                                                  \item we give, via a bijective proof, an interpretation of~\eqref{pnom} as number of specified colored directed lattice paths (Theorem~\ref{lattice}).
                                                                                                                                                  \item we propose an interpretation in terms of spin chain systems (Proposition~\ref{spin}).
                                                                                                                                                \end{enumerate}
We discuss, also in this section, some important examples of arithmetical polynomial triangles related to restricted occupancy models.
  \item In Section \ref{s3}, we give extensions of several binomial identities (Table~\ref{tab1}, Identities~\ref{id4},~\ref{squares}), and investigate some algebraic properties of the mapping $\a \mapsto \mathsf{T}(\a)$. The techniques are elementary and the proofs are straightforward, but the results seem interesting in their own right.
  \item In Section \ref{s5}, we recursively calculate generating functions of the column sequences of $\mathsf{T}(\a)$ (Proposition~\ref{cgf}). Interesting corollaries follow from these recurrence relations such as new formulae for the Fibonacci numbers (Corollary~\ref{fibo}) and Hermite polynomials (Corollary~\ref{hermite}) in terms of trinomial coefficients.
  \item We introduce in Section \ref{s6} the notion of entropy density function in the thermodynamical limit (that is when $k$ and $n$ go to infinity and the ratio $n/k$ is fixed) and study its properties (Theorem~\ref{ent}).
\end{itemize}

\section{ Combinatorial interpretations}\label{s2}
Besides GFS, we give in this section three combinatorial models
enumerated by arithmetical polynomial coefficients ($a_i \in \mathbbm{N}$). Our proofs are mainly bijective. To do this, we begin by presenting our own approach to the GFS.

\subsection{Restricted occupancy model} \label{model1} Consider a ball-in-box model in which $n$ \emph{undistinguishable}
balls (particles) are distributed among $k$ \emph{distinguishable} boxes (states).
Without restrictions on the boxes's occupancies, this model is known
as Bose-Einstein statistics. If one object at most is
allowed to occupy a box, the model  is called Fermi-Dirac statistics. In
the more general problem, there is a maximum and minimum number of
balls that any box can contain. Restricted occupancy models have
many applications~\cite{charala}. For instance, the Gentile-Freund model, where one allows
at most $m$ balls to fall in each box ($m\geq 1$), has been applied to the analysis
of socioeconomic and transport systems (see, e.g,~\cite{kapur}).

Assume a configuration in which $k_{i}$ boxes among $k$ ones are
occupied by $i$ balls, for each $i=1, \ldots m$. The number of
vacant boxes is obviously $k_{0}=k-\sum_{i=1}^{m}k_{i}$ and the
total number of balls is $n=\sum_{i=1}^{m}i k_{i}$. Then there are
${{k} \choose k_0, k_1, \ldots, k_m}=\nicefrac{k!}{k_{0}!k_{1}! \cdots
k_{m}!}$ ways to realize a configuration. The total number of
ways to put the $n$ balls in the boxes is obtained by summing over all
$(m+1)$--tuples of non-negative integers
$(k_{0},k_{1},\cdots,k_{m})$ subject to $\sum_{i=0}^mk_i=k$ and $\sum_{i=1}^m ik_i=n$. Then by using the
multinomial formula, one easily finds that the ordinary generating
function of this number is the $k$--th power of the polynomial
$1+t+t^2+ \cdots +t^m$. That is the polynomial coefficients
associated with the vector $a_i=[i \leq m]$\footnote{Throughout the paper, we use the Iverson bracket notation to indicate that $[P]=1$ if the proposition $P$ is true and 0 otherwise.}. These coefficients have well-known properties~\cite{Boll1,Bonda,Freund}.

As is clear from~\eqref{pnom1}, the polynomial coefficients associated with a general vector $\a$
count the total number of distributions of the balls in the above
restricted occupancy model; but each configuration is now weighted
by the monomials $\a^{\mathbf{k}}$.

We note, by the way, a more elegant form of~\eqref{pnom1} in terms
of weighted restricted partitions of $n$ :

\begin{lemma} \label{prop1} The polynomial coefficient~\eqref{pnom} can be written
as \be \label{t} \pnom{k}{n}{\a}=\sum_{{\lambda \vdash n \atop
l(\lambda')\leq m}} a_0^{k-l(\lambda)} h(\lambda)w_{\a}(\lambda) {k \choose
l(\lambda)},\ee where $\lambda \vdash n$ indicates that the sum
runs over all partitions of $n$ : $\lambda = (1^{k_{1}}2^{k_{2}} \ldots \, m^{k_{m}})$ whose greatest part do not exceed $m$,
symbolized by $l(\lambda')\leq m$, $\lambda'$ being the conjugate
partition of $\lambda$; $h$ is the function \[
h(\lambda)={{l(\lambda)}\choose k_{1},k_{2}, \cdots ,k_{m}}, \]
$l(\lambda)= \sum_{i=1}^{n}k_{i}$ is the length of the partition
$\lambda$. $w_{\a}$ is a function that assigns to $\lambda$ the
weight \[ w_{\a}(\lambda) =
\prod_{i=1}^m a_i^{k_i}. \]
\end{lemma}
\begin{proof}
Identify a configuration in which $k_{i}$ boxes among $k$ ones are
occupied by $i$ balls, $i=0, \ldots m$, with a restricted
partition $\lambda=(1^{k_{1}}2^{k_{2}} \cdots
\, m^{k_{m}})$ of the total number of balls $n=\sum_{i=1}^m i k_i$. To such a partition, one can attach a Ferrers
diagram where the number $k_{i}$ represents the multiplicity of
rows with $i$ dots and the length of the first row is less than or
equal to $m$. Then the sum in~\eqref{pnom1} becomes over all
restricted partitions of $n$. The expression~\eqref{t} follows
by replacing $k_0$ by $k-l(\lambda)$ and rewriting the multinomial
coefficient as $h(\lambda) {k \choose l(\lambda)}$. \end{proof}

Regrouping the terms with $l(\lambda)=i$, we get the following useful form \be \label{Tbis} \pnom{k}{n}{\a}=\sum_{i=\lceil \frac{n}{m}\rceil}^n a_0^{k-i} \alpha_{n,i} {k \choose i}, \ee where \be \label{alph} \alpha_{n,i}=\sum_{{\lambda \vdash n\; ,\; l(\lambda') \leq m \atop l(\lambda)=i}} h(\lambda) w_{\a}(\lambda)=\pnom{i}{n-i}{\a_1},\ee and  $\a_1=(a_1, \ldots ,a_m)$.

We can interpret the summands in the formula~\eqref{t} as
follows~: the binomial coefficient is the number of ways to
collect $l(\lambda)$ non-vacant boxes among $k$ ones. A
configuration $\lambda$ being fixed, this number should be
multiplied by the number $h(\lambda)$ of ways to arrange $k_1$
boxes with 1 ball, $k_2$ boxes with 2 balls, \ldots, $k_m$ boxes with
$m$ balls in a sequence of length $l(\lambda)$. The result is
finally weighted by the monomial $w_{\a}(\lambda)$ which consists
of the product of the weights of every row in the Ferrers
diagram~: each row with $i$ dots has $a_i$ colors, including the
$k_0$ ``empty'' rows.

\begin{remark} The number of permitted configurations is the number of
partitions of $n$ which fit inside a $k \times m$ rectangle. It is
given by the coefficient of $q^n$ in the gaussian polynomial
$\left[ {k+m \atop k} \right]_q$~\cite{Andrews4}.  To illustrate, let us consider 5 balls to be thrown into 4 boxes
with maximal occupancy 3. There are $[q^5]\hskip-4pt\left[
{7 \atop 4} \right]_q = 4$ possible configurations : $(2 3), (1^2  3), (1  2^2), (1^3  2)$. Summing the contributions of these configurations yields $\pnom{4}{5}{\,\a}=12 a_0^2 a_2 a_3 +12 a_0 a_1^2 a_3 +12 a_0 a_1 a_2^2+4 a_1^3 a_2 $. \end{remark}

For non-negative integers $a_i$, the interpretation of
\eqref{pnom} suggests that $a_i=\enom{1}{i}{\a}$ can
be regarded as the number of ways to throw $i$ balls in one box.
To make this proposal plausible, one could consider that the
interior of each box is discretized, i.e., balls live in
cells whose occupation can be either single or multiple. This
intrinsic ``one-box structure'' was proposed by
Fang~\cite{fang1,fang2} who, in an essentially probabilistic
approach to the Gentile-Freund model, considered the case where
each box contains $m$ cells and the balls are assigned to the boxes
in such a way that no cell can accommodate more than one ball.
Obviously, according to this picture, if $a_j=0$ for some index
$j$, then $j$ balls are not allowed to lodge together in a same box.
By reason of the above interpretation, the vector $\a =(a_0,a_1,
\ldots, a_m)$ will be called the \emph{color vector} of the
triangle $\mathsf{T}(\a)$.

\subsection{Score in drawing colored balls} \label{model2}
Let us consider the following game of chance : suppose that a box
contains $N$ balls labeled by numbers from 0 to $m$ and assume we
have $a_i$ balls with label (or color) $i$, $N=\sum_{i=0}^m a_i$. A
ball is repeatedly drawn at random and put back in the box, with all balls having equal
chances of being chosen at any time. Suppose that, in each trial, the capital of a player is increased by
$j$ when ball number $j$ shows up with probability $a_j/N$.
Let $g_i$ denote the gain in the $i$-th trial and let $G_k=g_1+g_2+
\cdots +g_k$ be the partial gain at time $k$. The probability
generating function of the random variable $G_k$ is \[
\frac{1}{N^k}\left(a_0+a_1
t+a_2 t^2+\cdots+a_{m}t^{m}\right)^k .\] So, its probability mass
function reads \be \label{proba}
\mathbb{P}(G_k=n)=\frac{1}{N^k}\pnom{k}{n}{\a}=
\frac{1}{\sum_{i=0}^{mk} \pnom{k}{i}{\a}}\pnom{k}{n}{\a}\; , \quad
n=0,1 \ldots ,mk. \ee where $\a=(a_0,a_1,\ldots,a_m)$. This is the
probability that the player shows $n$ points after $k$ draws. Note
that the denominator in~\eqref{proba}, i.e, the total number of
possibilities, is the $k$-th row sum in the $\a$-triangle. Since,
moreover, the space of all possibilities is endowed with the uniform
probability measure, we have
\begin{theorem} \label{score}
Let $\a=(a_0,a_1,\ldots,a_m) \in \mathbbm{N}^{m+1}$. The polynomial
coefficient associated with $\a$ is the number of
ways to record $n$ points after $k$ trials in the above described
drawing game, the integer $a_j$ being the number of balls of color
$j$.
\end{theorem}

Obviously, if $m=1$, the variables $g_i$ are $(0,1)$ Bernoulli random variables
and therefore $G_k$ has binomial distribution. In this sense, the
probability mass function~\eqref{proba} is seen as a generalization of the
binomial distribution based on trials with more than 2 outcomes, provided
the variables $g_i$ are not decomposable to Bernoulli ones.

As noted in the introduction, the colorless version of the distribution ~\eqref{proba} was first arrived at by De Moivre. In the middle of the twentieth century, it was restudied by Steyn as the limit of a generalization of the
hypergeometric distribution~\cite{steyn}. It was also investigated by the authors of~\cite{pana}, where the distribution is termed  ``cluster binomial distribution'' because of the multitude of outcomes for one trial. It was also studied in detail by the authors of\cite{balasu}.

\subsection{Directed lattice paths}
In probability theory, it is well-known that the evolution of sums
of independent discrete random variables, like that of the last
model, can be described by \emph{lattice paths}. We plan to make use
of this fact to propose the third model.

In this section, by a directed lattice path we mean a polygonal
line of the discrete Cartesian half plane $\mathbbm{N} \times
\mathbbm{Z}$ whose ``direction of increase'' is the horizontal axis
and the allowed steps are simple, i.e, of the form $(1,s)$ with $s \in
\mathbbm{Z}$~\cite{Bander}.

We have the following combinatorial interpretation of the polynomial coefficients :

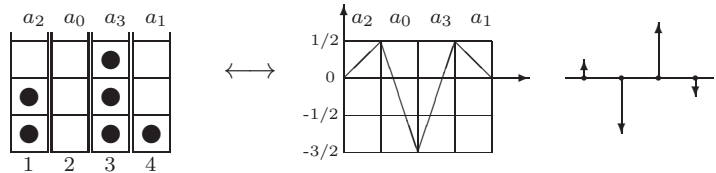
\begin{figure}
\begin{center}
  \setlength{\unitlength}{0.7pt}
\begin{picture}(280,100)(0,-10)

\put(-40,0){\line(0,1){65}}
\put(-20,0){\line(0,1){65}}
\put(-30,10){\circle*{10}}
\put(-30,30){\circle*{10}}
\multiput(-40,0)(0,20){4}{\line(1,0){20}}
\put(-37.5,70){\scriptsize{ $a_2$}}
\put(-37.5,-10){\scriptsize{ $1$}}
\put(-18,0){\line(0,1){65}}
\put(2,0){\line(0,1){65}}
\multiput(-18,0)(0,20){4}{\line(1,0){20}}
\put(-15.5,70){\scriptsize{ $a_0$}}
\put(-15.5,-10){\scriptsize{ $2$}}
\put(4,0){\line(0,1){65}}
\put(24,0){\line(0,1){65}}
\put(14,10){\circle*{10}}
\put(14,30){\circle*{10}}
\put(14,50){\circle*{10}}
\multiput(4,0)(0,20){4}{\line(1,0){20}}
\put(6.5,70){\scriptsize{ $a_3$}}
\put(6.5,-10){\scriptsize{ $3$}}
\put(26,0){\line(0,1){65}}
\put(46,0){\line(0,1){65}}
\put(36,10){\circle*{10}}
\multiput(26,0)(0,20){4}{\line(1,0){20}}
\put(28.5,70){\scriptsize{ $a_1$}}
\put(28.5,-10){\scriptsize{ $4$}}
\put(75,40){$\longleftrightarrow$}
\thinlines\put(140,40){\vector(1,0){100}}
\put(140,0){\vector(0,1){80}}
\linethickness{0.001mm}
\multiput(140,0)(20,0){5}{\line(0,1){60}}
\multiput(140,0)(0,20){4}{\line(1,0){80}}
\put(130,38){\tiny{0}}
\put(122,58){\tiny{1/2}}
\put(118,18){\tiny{-1/2}}
\put(118,-2){\tiny{-3/2}}
\linethickness{2mm}
\put(140,40){\line(1,1){20}}
\put(160,60){\line(1,-3){20}}
\put(180,0){\line(1,3){20}}
\put(200,60){\line(1,-1){20}}
\put(140,70){\scriptsize{ $a_2$}}
\put(160.5,70){\scriptsize{ $a_0$}}
\put(182.5,70){\scriptsize{ $a_3$}}
\put(204.5,70){\scriptsize{ $a_1$}}
\linethickness{0.1mm}
\put(260,40){\line(1,0){80}}
\put(270,40){\vector(0,1){10}}
\put(290,40){\vector(0,-1){30}}
\put(310,40){\vector(0,1){30}}
\put(330,40){\vector(0,-1){10}}
\put(270,40){\circle*{3}}
\put(290,40){\circle*{3}}
\put(310,40){\circle*{3}}
\put(330,40){\circle*{3}}
\end{picture}
\end{center}
\caption{Illustration of the bijection~\eqref{bijection} for $m=3$. In this example, the shape is $(2,0,3,1)$ and the slops are $(\nicefrac{1}{2},-\nicefrac{3}{2},\nicefrac{3}{2},-\nicefrac{1}{2})$ .  The associated spin chain is also displayed.}
\label{fig2}
\end{figure}

\begin{theorem} \label{lattice}Let $\mathcal{S}^{(\a)}(k,n)$ denote the set of
lattice paths of length $k$ starting from the origin, ending in
the point with coordinates $(k,n-m k/2)$ and using the
steps $s_i=(1,i-m/2)$, $i=0,\ldots,m$ ; step $s_i$ coming in
$a_i$ colors. Then \[ \pnom{k}{n}{\a} = \# \, \mathcal{S}^{(\a)}(k,n).\]
\end{theorem}
\begin{proof} We will set up a bijection from the set of occupancy shapes of the GFS
onto the set $\mathcal{S}^{(\a)}(k,n)$. Consider $k$ boxes numbered
from 1 to $k$ and arranged from left to right in ascending order as
illustrated in Figure~\ref{fig2} for $k=4$. Assume box n$^{\circ} i$
is occupied by $n_i$ balls. The occupancy shape of this
configuration is the $k$-tuple $(n_1, \ldots, n_k)$;
$n_1+\ldots+n_k=n$ and $n_i \leq m$. An occupancy shape can then be
regarded as a restricted composition of $n$. We construct our
bijection as follows. With the shape $(n_1, \ldots, n_k)$ we
associate a lattice path of $k$ steps starting from the origin in
such a way that to box n$^{\circ} i$ we assign bijectively a simple
step $s_i$ with slope $n_i-m/2$: \be \label{bijection} (n_1,
\ldots, n_k) \leftrightarrow \left(s_1, \ldots, s_k\right), \quad
s_i=\left(1,n_i-m/2\right).\ee Because box n$^{\circ} i$ can have $a_i$
colors, the corresponding step $s_i$ can appear with so many
incarnations. Moreover, since in a given configuration $\lambda
=(0^{k_0}1^{k_1} \cdots m^{k_m})$, $k_i$ boxes accommodate $i$
balls, there are $k_i$ steps $s_i$ in the corresponding lattice
path. Thus, the latter ends in the altitude $\sum_{i=1}^k (n_i-m/2)=n-m k/2$. This ends the proof.\end{proof}

 If $mk$ is even then there exist a path that ends in the $x$-axis, i.e, a bridge if we use the terminology of~\cite{Bander} (cf. Figure~\ref{fig2}). This case concerns an occupancy model with half-filling, counted by  the central polynomial coefficient $\pnom{k}{\nicefrac{mk}{2}}{\; \a}$. For further lattice paths interpretations of central trinomial coefficients, see the study of David Callan~\cite{callan}.

\subsection{A spin chain model} Consider a chain of $k$ sites; each site is occupied by a particle with spin $\nicefrac{m}{2}$. The $m+1$ components of the spin runs over the set  $\{-\nicefrac{m}{2}, -\nicefrac{m}{2}+1, \ldots, \nicefrac{m}{2}\}$. As for the Ising model, define the ``\emph{magnetization}'' of a spin configuration of the system as the sum of spin projections divided by $k$ : \[ \frac{\hbox{sum of up spins} \uparrow - \,  \left|\hbox{sum of down spins} \downarrow \right|}{k}.\] Identifying the slope $n_i-m/2$ of the $i$-th step in the lattice path model with the spin projection $n_i-m/2$, as illustrated in Figure~\ref{fig2},  we have the following interpretation : \begin{proposition} \label{spin} The polynomial coefficient associated with the vector $\a$ is the number of spin configurations with magnetization $n/k-m/2$; spin projection $n_i-m/2$ comes in $a_i$ colors. \end{proposition}

The half-filling occupation discussed in the end of the last subsection concerns now the spin configuration with vanishing magnetization.
\subsection{Examples of combinatorial polynomial triangles}
Let us now discuss some instances of arithmetical polynomial triangles associated with specific color vectors.\begin{example}[\emph{Polynomial triangle associated with binomial coefficients}]  Let $a_i={m \choose i}$. Here, the polynomial coefficients reduce to
the binomial coefficients ${mk \choose n}$, i.e., the $n$ balls
are distributed into $m$ copies of $k$ boxes according to
Fermi-Dirac statistics. In this case, equation~\eqref{pnom1} leads to the binomial formula
\[ {mk\choose n}
  = \sum_{\scriptstyle{k_1+2k_2+\cdots+m k_m = n} \atop \scriptstyle{k_0+k_1+\cdots+ k_m = k}} \prod_{l=0}^m
    {\sum_{i=l}^m k_i \choose k_l} {m\choose l}^{k_l}. \] \end{example}
\begin{example}Let $a_i=a_i [p \leq i \leq m]$, $a_p \neq 0$. In this case we have weighted GFS where no boxes occupied by less than $p$ balls are
permitted. Here the sum~\eqref{t} reduces to a sum over
partitions such that $p \leq l(\lambda') \leq m$. The number of
possible ways is readily found to be \[
\pnom{k}{n}{\a}=\pnom{k}{n-kp}{\a_{p}} \quad \mbox{if} \quad kp \leq n
\qquad \mbox{and} \qquad  \pnom{k}{n}{\a} = 0 \quad \mbox{if} \quad
kp > n. \] where $\a_p = (a_p,\ldots,a_m)$. In particular, the
number of ways in which $n$ unlabeled objects can be distributed in
$k$ uncolored labeled boxes allowing at most $m$ objects and at
least $p$ objects to fall in each box, is $\enom{k}{n-k p}{\tiny{m-p}}$.  This is understandable, since we have to fill each box with $p$
objects to guarantee the minimum occupancy level and distribute the
remaining $n-k p$ ones among $k$ boxes allowing at most $m-p$
balls per box. We notice that for $p=1$, $\enom{k}{n-k}{m-1}$ is also the number of compositions of
$n$ with exactly $k$ parts, each less than or equal to $m$~\cite[p.55]{Andrews4}. \end{example}
\begin{example}[\emph{Restricted occupancy model with distinguishable balls}] Let $a_i=1/i! \; [0 \leq i \leq m]$. This case is the exponential version of the GFS: \[\pnom{k}{n}{\a}=[t^n]\left(e_m(t)\right)^k,\] where $e_m(t)$ is simply the $m$th section of the exponential series.. When $m=\infty$ the coefficient $n! \pnom{k}{n}{\a} = k^n$ is the number of ways in which $n$ \emph{distinguishable} balls can be thrown in $k$ distinguishable boxes (The so-called Maxwell-Boltzmann statistics). For finite $m$, the integer $ n!\enom{k}{n}{\a}$ enumerates the same occupancy model but with the restriction that no more than $m$ labeled balls can lodge in the same box. For $m=2$, the triangle is recorded as A141765.

 If $a_i=1/i! \; [1 \leq i \leq m]$, the integer  \[ n!\pnom{k}{n}{\a} = n! [t^n]\left(e_m(t)-1\right)^k ,  \quad (k \leq n \leq m k) \] is the statistical weight of the above restricted Maxwell-Boltzmann model allowing that no box is left unoccupied. For $m=\infty$, we have that \[\frac{n!}{k!} \pnom{k}{n}{\a}={n \brace k},\] ${n \brace k}$ being the Stirling numbers of the second kind. Therefore, for finite $m$, the integer $n!/k! \pnom{k}{n}{\a}$ is the number of ways of partitioning a set of  $n$ elements into  $k$ nonempty subsets with the restriction that no subset can contain more than $m$ elements. In the particular case $m=2$, we have \[ \frac{n!}{k!} \pnom{k}{n}{(0,1,{1\over 2})}=\frac{n!}{2^{n-k}(2 k-n)! (n-k)!} ,\]  which are the coefficients of the so-called Bessel polynomials : \[\sum_{n=k}^{2k}  \frac{n!}{k!} \pnom{k}{n}{(0,1,{1\over 2})}t^{n-k} =\sum_{i=0}^k\frac{(k+i)!}{(k-i)!i!}\,\left(\frac{t}{2}\right)^i = y_k(t).\] These numbers have been studied by Choi and Smith~\cite{choi1,choi2}.
\end{example}

\section{Polynomial coefficient Identities} \label{s3} Binomial coefficients satisfy an amazing plethora of dazzling identities. It is natural to seek their extensions to the polynomial case. In this section we demonstrate generalizations of some of the binomial coefficient identities.

\begin{table}
\small \begin{tabular}{lcc}
\hline \hline
\textbf{Identity} & \textbf{Binomial} & \textbf{Polynomial}\\
\hline &  &\\
Factorial expansion &  $\displaystyle{k \choose n}=\frac{k!}{n!(k-n)!}$ & Equation~\eqref{pnom1} \\
& & \\
Symmetry$^\dag$ &$\displaystyle\binom{k}{n}=\binom{k}{k-n}$ & $\displaystyle\displaystyle \pnom{k}{n}{\a}=\pnom{k}{m k-n}{\mathrm{J} \a}$ \\
& & \\
Absorption/Extraction  &$\displaystyle \binom{k}{n}=\frac{k}{n}\binom{k-1}{n-1}$ & $ \displaystyle \pnom{k}{n}{\a} = \frac{k}{n} \sum_{i=1}^m i a_i \pnom{k-1}{n-i}{\a}$\\
&  &\\
Vandermonde convolution  &$ \displaystyle \sum_{i+j=n}\binom{r}{i}
\binom{s}{j}= \binom{r+s}{n}$ & $\displaystyle
\sum_{i+j=n}\pnom{r}{i}{\a}
\pnom{s}{j}{\a}= \pnom{r+s}{n}{\a} $ \\
&  &\\
Addition/Induction  &$\displaystyle \binom{k}{n}=\binom{k-1}{n}+\binom{k-1}{n-1}$  & $\displaystyle \pnom{k}{n}{\a} = \sum_{i=0}^{m} a_i \pnom{k-1}{n-i}{\a}$ \\
& & \\
Binomial theorem  &$\displaystyle \sum_{n=0}^k \binom{k}{n} x^n y^{k-n}=(x+y)^k$  & $\displaystyle \sum_{n=0}^{m k} \pnom{k}{n}{\a} x^n y^{m k-n}=\left(\sum_{i=0}^m a_i x^i y^{m-i}\right)^k $ \\
& & \\
Upper summation$^\ddag$ &$\displaystyle \sum_{0 \leq k \leq p}{k \choose n}={p+1 \choose n+1}$ & $\displaystyle \sum_{0 \leq k \leq p} \frac{1}{a_0^k}\pnom{k}{n}{\a}=\sum_{i=\lceil \frac{n}{m}\rceil}^n \frac{1}{a_0^i} \alpha_{n,i} {p+1 \choose i+1}$\\
& &\\
Upper negation & $\displaystyle
\binom{k}{n}=(-1)^n \binom{n-k-1}{n}$ & $\displaystyle
\pnom{k}{n}{\a} = \sum_{i=0}^n
a_0^{k-i} (-1)^i \alpha_{n,i}\binom{i-k-1}{i} $  \\
&&\\
${\mbox{A recurrence relation} \atop \mbox{with respect to $n \geq 1$} }$& $\displaystyle \binom{k}{n}=\frac{k+1-n}{n} \binom{k}{n-1}$ & $\displaystyle \pnom{k}{n}{\a}=\frac{1}{n a_0} \sum_{l=1}^{m}\left((k+1)l-n\right) a_l \pnom{k}{n-l}{\a}$\\
&&\\
\hline \hline
\end{tabular}
\medskip
 \caption {Extensions of eight of the top ten binomial coefficients identities. \newline \indent  $^{\dag}$\, $\mathrm{J}$  is the $(m+1) \times (m+1)$ backward identity matrix, that is matrix with 1's on the anti-diagonal and 0's elsewhere.  \newline \indent $^{\ddag}$  The coefficients $\alpha_{n,i}$ are defined in~\eqref{alph}.}
\label{tab1}
 \end{table}
\subsection{Extension of Top Ten binomial identities}  In Table~\ref{tab1} we propose polynomial extensions of some of the top ten binomial identities displayed in table 174 of the ``concrete Mathematics'' by Graham, Knuth, and Patashnik~\cite{graham}. The proof of these generalizations is straightforward. \begin{proof}(sketch) The polynomial symmetry relation is readily established by writing that $p_{\mathrm{J}\a}(t)$ is the reciprocal polynomial of $p_\a(t)$ : \[p_{\mathrm{J}\a}(t)=\sum_{i=0}^m a_{m-i}t^i=t^{m}p_{\a}(t^{-1})\] and \[ \pnom{k}{mk-n}{\mathrm{J}\a}= [t^{mk-n}]t^{mk}p_\a(t^{-1})^k =[t^{-n}]p_\a(t^{-1})^k= \pnom{k}{n}{\a}.\] The Absorption/Extraction property follows by taking the derivative of both sides of $p_\a(t)^k=\sum_{n=0}^{mk} \pnom{k}{n}{\a}t^n$ with respect to $t$, and equating the coefficients of $t^n$. The Vandermonde convolution is usually obtained by equating coefficients on both sides of
$p_{\a}^{r+s}(t)=p_{\a}^r(t) p_{\a}^s(t)$. The Addition/ Induction relation is a particular case of Vandermonde convolution with $r=1$ and
$s=k-1$. The generalized binomial theorem is an obvious consequence of definition~\ref{def1}. To prove the generalized Upper summation and Upper negation identities, we apply the binomial upper summation and upper negation to the binomial coefficients in right-hand-side of~\eqref{Tbis}.  Finally, the recurrence with respect to $n\geq 1$ is a particular case of a recurrence for powers of Taylor series (see~\cite{gould1} and references therein). To prove it take logarithms of the equation $\left(p_\a(t)\right)^k=\sum_n \pnom{k}{n}{\a}t^n$ and differentiate with respect to $t$ and equate the coefficients of $t^n$ on both sides of the obtained equation~\cite{gould1}. \end{proof}

 \paragraph{\textbf{Remarks on the polynomial symmetry}.} Two points about the polynomial symmetry are worthy of note :
 \begin{itemize}
   \item If $\a$ is a palindrome, i.e, $a_i=a_{m-i} \; \forall \, i = 0, \ldots , m$, then $\a = \mathrm{J} \a$ and thus, through the generalized symmetry relation, the \textit{centered} triangle $\mathsf{T}(\a)$ is mirror symmetric across the median column (cf. Figure~\ref{fig}).
   \item In physics literature, the term ``holes'' is used to designate missing occupancies in
ball-in-box models and the notion of particle-hole duality implies that instead of studying particles, one can get similar information by
studying the holes. We infer that the generalized symmetry relation provides a
particle-hole duality. To see this, consider the restricted
occupancy model discussed in subsection~\ref{model1} and assume,
following Fang~\cite{fang1,fang2}, that each box contains
$m$ cells; no more than 1 ball can lodge in the same cell. For our
purpose, a particle is simply an occupied cell while a hole is
identified with an empty one. Therefore, it is clear that if $n$
balls (particles) are distributed among $k$ boxes (states), there are
$mk-n$ holes. According to this picture, we learn from the
symmetry relation that a system of $n$ particles governed by GFS
associated with the color vector $\a$ can equivalently be
described by $mk-n$ missing particles obeying GFS associated with
the color vector $\mathrm{J}\a$. This equivalence is just the
particle-hole duality. Particularly, if $\a$ is palindromic, (in this case the polynomial $p_\a$ is self-reciprocal)
particles and holes obey the same statistics. In the point of view of lattice paths, this duality acts as a simple reflection about the $x$-axis.
 \end{itemize}

 \subsection{More identities}
 \paragraph{\bf A pretty identity.} \be \label{id4} \sum_{l} \pnom{r}{p+l}{\mathrm{J} \a}
\pnom{s}{n+l}{\a} = \pnom{r+s}{m r-p+n}{\a}=\pnom{r+s}{m s+p-n}{\mathrm{J} \a}.\ee

\begin{proof} Follows from the symmetry relation and the application of
Vandermonde convolution. \end{proof}

If $p=n=0$, The identity~\eqref{id4} can be cast into the following matrix form
\[ \mathsf{T}(\mathrm{J} \a) \cdot  \mathsf{T}(\a)^{\mathrm{t}}= \mathsf{S}(\a)=\mathsf{S}(\mathrm{J} \a)^{\mathrm{t}},\] where $\mathsf{S}(\a)$ is the array whose $(r,s)$-entry  is $\pnom{r+s}{m r}{\a}$. Obviously, the matrix $\mathsf{S}(\a)$ is symmetric if $\a$ is a palindrome, and generalizes the familiar Pascal symmetric matrix. For instance, $\mathsf{S}((1,1,1))$ begins as
\[ \mathsf{S}((1,1,1))=\begin{pmatrix} 1 & 1 & 1
& 1 & 1 & 1 & \ldots \\1 & 3 & 6 & 10 & 15 & 21 & \ldots \\ 1 & 6 & 19
& 45 & 90 & 161 & \ldots \\ 1 & 10 & 45 & 141 & 357 & 784 & \ldots \\ 1 & 15 &
90 & 357 & 1107 & 2907 & \ldots
\\1 & 21 & 161 & 784 & 2907 & 8953 &\ldots\\ \vdots & \vdots & \vdots & \vdots & \vdots
& \vdots & \ddots \end{pmatrix}.\]

Putting in~\eqref{id4} $r=s=k$ and $p=n=0$, we get the identity :
\paragraph{\bf Sum of squares.} \emph{If $\a$ is a palindrome then} \be  \label{squares}\sum_{n=0}^{mk} \pnom{k}{n}{\a}^2
= \pnom{2k}{m k}{\a}. \ee
This identity generalizes the well-known formula $\sum_{n=0}^{k} \binom{k}{n}^2
= \binom{2k}{k}$,~\cite[p.78]{benj}.

\subsection{Some algebraic facts about the mapping $\a \mapsto \mathsf{T}(\a)$}
Now we prove some general algebraic properties.
\begin{fact}  \emph{Product of two polynomial triangles}. As infinite matrices, the
product of two polynomial triangles is a polynomial triangle,
i.e., \be \label{id3} \mathsf{T}(\a) \cdot \mathsf{T}(\b)=\mathsf{T}(\a \circ \b)\ee where for $\a =(a_0,\ldots,a_m)$,
and $\b=(b_0,\ldots,b_p)$,  $\a \circ\b$ is the $(mp+1)$-vector whose $i$-th component is given by
\be\label{op} (\a \circ\b)_i=\sum_{j=0}^m a_j \pnom{j}{i}{\,\b}, \quad i=0, \ldots ,mp.\ee
\end{fact}
\begin{proof}The generating polynomial of $\a \circ \b$ is \[p_{\a \circ \b}(t)= \sum_{i=0}^{mp}\left(\sum_{j=0}^m a_j \pnom{j}{i}{\b}\right) t^i = \sum_{j=0}^m a_j \sum_{i=0}^{j p}\pnom{j}{i}{\b} t^i  = \sum_{j=0}^m a_j p_\b(t)^j=p_\a (p_\b(t)).\] Hence the $(k,n)$-entry  of $\mathsf{T}(\a \circ \b)$ can be written as \[\pnom{k}{n}{\a \circ \b} = [t^n]\left(p_\a(p_\b (t))\right)^k = \sum_{l=0}^{mk} \pnom{k}{l}{\a} [t^n] p_\b(t)^l =\sum_{l=0}^{mk}\pnom{k}{l}{\a}
\pnom{l}{n}{\b}. \] The last sum is exactly the $(k,n)$-entry of the array $\mathsf{T}(\a) \cdot \mathsf{T}(\b)$.
\end{proof} \noindent The formula~\eqref{id3} generalizes the identity $\sum_{l \geq 0}
{k \choose l}{l \choose n}= 2^{k-n} {k \choose n} $~\cite[p.78]{benj}.

The set of all vectors equipped with the operation~\eqref{op} is a monoid with identity element given by $\mathbf{e}=(0,1,0,\ldots)$; the triangle $\mathsf{T} (\mathbf{e})$ is the infinite identity matrix. The facts~\ref{fac2},~\ref{fac3},~\ref{fac4},~\ref{fac5} below are straightforward features of the binary operation~\eqref{op}. We leave their proofs as easy exercises for the reader.

\bfa \label{fac2} The operation $\circ$, which can be written in matrix form as
    \[ \a \circ \b =\mathsf{T}( \b)^\mathrm{t} \, \a , \] is associative and obviously linear with respect to the first argument, but not with respect to the second. However, one has the property
\[ \a \circ (\lambda \b)= (\mathrm{I}_\lambda \a) \circ \b,\] where, for every real number $\lambda$, $\mathrm{I}_\lambda = \mathrm{diag}(1,\lambda, \lambda^2, \ldots )$.
\efa
  \bfa \label{fac3} Let $\lambda$ be an arbitrary real number. Then
\bea \nonumber \mathsf{T}(\lambda \a)& =&\mathrm{I}_\lambda \cdot \mathsf{T}(\a).\\
\nonumber \mathsf{T}(\mathrm{I}_\lambda \a)&=&  \mathsf{T}(\a) \cdot \mathrm{I}_\lambda.\eea \efa

A particularly interesting case is that of binomial triangles ($m=1$) for which the map $\mathsf{T}(\cdot)$ is a group isomorphism:
\bfa \label{fac4}  The triangle $\mathsf{T}((a_0,a_1))$ is a Riordan array~\cite{riordan} :
\[ \mathsf{T}((a_0,a_1))= \left(\frac{1}{1-a_0 t} , \frac{a_1 t}{1-a_0 t}\right),\] and the set of these arrays :  \[\mathcal{T}_2 :=\{\mathsf{T}(\a) \, | \, \a \in \mathbbm{R} \times \mathbbm{R}^* \}\] is a subgroup of the Riordan group isomorphic to $(\mathbbm{R} \times \mathbbm{R}^*, \circ)$. \efa
\bfa \label{fac5}
The group $(\mathcal{T}_2,\cdot)$ is isomorphic to the $``a x +b''$ group generated by affine transformations of the real line. \efa \noindent Recall that the $``a x +b''$ group, generated by all dilations and translations of the real line, is isomorphic to the multiplicative group of all (real) $2 \times 2$ matrices of the form~\cite[p. 35]{sternberg}:
\[\begin{pmatrix} a&b\\0&1 \end{pmatrix} \, ; \quad a \neq 0 .\]


\section{Generating functions} \label{s5} In this section, we prove several properties of the generating functions of columns  of $\mathsf{T}(\a)$.

Let $\mathscr{F}_n(u)$ and $\mathscr{E}_n(u)$ be respectively the ordinary and the
exponential generating functions for the $n$-th column of the
(left-justified) triangle $\mathsf{T}(\a)$. Then we have

\begin{proposition} \label{cgf}
The generating functions $\mathscr{F}_n$ and $\mathscr{E}_n$ take the forms \be \label{GF}
\mathscr{F}_n(u)=\frac{P_n^{(\a)}(u)}{\left(1-a_0 u\right){}^{n+1}}  \quad
and \quad \mathscr{E}_n(u)= e^{a_0 u} R_n^{(\a)}(u), \ee where \[ P_n^{(\a)}(u) =
\sum_{i=\lceil \frac{n}{m}\rceil}^n \alpha_{n,i} u^i(1-a_0 u)^{n-i} \quad \mbox{and} \quad R_n^{(\a)}(u) =
\sum_{i=\lceil \frac{n}{m}\rceil}^n \alpha_{n,i} \frac{u^i}{i!};\] $\alpha_{n,i}$ is defined by~\eqref{alph}. Moreover, the polynomials $P_n^{(\a)}(u)$ and $R_n^{(\a)}(u)$ are subject to the recursive equations \bea
 \label{recurdiff1} P_n^{(\a)}(u) &=& u \sum_{i=1}^m a_i \left(1-a_0 u\right)^{i-1} P_{n-i}^{(\a)}(u) \\
  \label{recurdiff2} \frac{\partial R_n^{(\a)}}{\partial u}(u) &=& \sum_{i=1}^{m} a_i
  R_{n-i}^{(\a)}(u),\eea with the initial conditions $P_0^{(\a)}(u)=R_0^{(\a)}(u)=1$
and $P_n^{(\a)}(u)=R_n^{(\a)}(u)=0$ for $n < 0$.
\end{proposition}

\begin{proof}
Using the form~\eqref{Tbis}, we have \[\mathscr{F}_n(u)=\sum_{k=0}^\infty
\pnom{k}{n}{\a}u^k = \sum_{i=\lceil \frac{n}{m}\rceil}^n a_0^{-i} \alpha_{n,i}
\sum_{k=0}^\infty {k \choose i} (a_0 u)^k.\] Employing the generating
function of binomial coefficients~:~$\sum_{k=0}^\infty {k \choose
n}u^k = u^n/(1-u)^{n+1}$, we find that $\mathscr{F}_n(u)$ can be
displayed in the form~\eqref{GF} with \[ P_n^{(\a)}(u)= \sum_{i=\lceil \frac{n}{m}\rceil}^n \alpha_{n,i} u^i(1-a_0 u)^{n-i}.\]

As for the expressions of $\mathscr{E}_n(u)$ and $R_n^{(a)}$, it results in the same way, using the exponential generating function
$\sum_{k=0}^\infty {k \choose n}u^k/k! = e^u u^n/n! $.

To prove the recursion relations~\eqref{recurdiff1} and~\eqref{recurdiff2} we use the Addition /Induction relation of Table~\ref{tab1}. It is immediate that for $n>0$ (recall $\pnom{0}{n}{\a}=0$ if
$n>0$)
\[\mathscr{F}_n(u)=\sum_{k=0}^\infty \pnom{k}{n}{\a}u^k =\sum_{k=1}^\infty \left( \sum_{i=0}^{m}
a_i \pnom{k-1}{n-i}{\a} \right) u^k = u \sum_{i=0}^m a_i
\mathscr{F}_{n-i}(u).
\] i.e, \[\left(1-a_0 u\right)\mathscr{F}_n(u)=u \sum_{i=1}^m
a_i \mathscr{F}_{n-i}(u),\] which yields the desired recurrence. On the other
hand, employing also the Addition/Induction relation, we find for
$n>0$
\[\mathscr{E}_n(u)= \sum_{k=1}^\infty
\pnom{k}{n}{\a}\frac{u^k}{k!} = \sum_{i=0}^{m} a_i \left(
\sum_{l=0}^\infty \pnom{l}{n-i}{\a} \frac{u^{l+1}}{(l+1)!}\right).\]
Taking the derivative of both sides, we find
\[\frac{\partial \mathscr{E}_n}{\partial u}(u)=\sum_{i=0}^m a_i \mathscr{E}_{n-i}(u),\]
from which equation~\eqref{recurdiff2} results strait-forwardly.
\end{proof}

 From~\eqref{GF}, the generating functions for the polynomials
$P_n^{(\a)}(u)$ and $R_n^{(\a)}(u)$ are  \bea
\sum_{n=0}^\infty P_n^{(\a)}(u) z^n &=& \frac{1-a_0u}{1-u
p_{\a}\left(\left(1-a_0u\right)z\right)}\;;\\
\label{egf} \sum_{n=0}^\infty R_n^{(\a)}(u) z^n &=&
\exp\left(u\left(p_{\a}(z)-a_0\right)\right). \eea

\subsection{The special case $m=2$}
If $m=2$, the two-term recurrence~\eqref{recurdiff1} can be explicitly solved by standard techniques. For the colorless case, this yields
\be \label{pm2} P_n^{(2)}(u)=\frac{\left(u+\sqrt{u(4-3 u)
   }\right)^{n+1}-\left(u-\sqrt{u(4-3 u)
   }\right)^{n+1}}{2^{n+1}\sqrt{u(4-3 u)}}.\ee The first few polynomials are
   \begin{center} \begin{tabular}{lll}
$n$ & &$ u^{-\lceil n/2 \rceil} P_n^{(2)}(u)$\\
  \hline
  0 && 1 \\
  1 &&1  \\
  2 && 1\\
  3 && $2-u$ \\
  4 && $1+u-u^2$ \\
  5 && $3-2u$ \\
  6 && $1 + 3 u - 4 u^2 + u^3$\\
  7 && $4-2 u-2 u^2+u^3$.\\
\hline
&&
\end{tabular}\end{center}
 From~\eqref{alph}, we derive $\alpha_{n,i} ={i \choose  n-i}$. Since $\sum_{i}{i \choose  n-i}=F_{n+1} $, where $F_{n}$ is the $n$-th Fibonacci number, we have $2^n P_n^{(2)}(1/2)=F_{n+1}$. Actually, we find the following appealing connection between Fibonacci numbers and trinomial coefficients : \begin{corollary} \label{fibo} For $n \geq 1$ \[ \sum_{k=\lceil (n-1)/2\rceil}^\infty \pnom{k}{n-1}{2} \,{1 \over 2^{k+1}} = F_{n}.\]
\end{corollary}
Moreover, from the generating function~\eqref{egf}, we see that $R_n^{(2)}(u)$ can be expressed as: \[ R_n^{(2)}(u)=(-1)^n \frac{\left(\sqrt{-u}\right)^n}{n!} H_n\left(\frac{\sqrt{-u}}{2}\right),\] where $H_n$ is the $n$-th Hermite polynomial~\cite{HMF}. As an interesting by-product of this connection, we find an expression of the Hermite polynomials in terms of colorless trinomial coefficients:
\begin{corollary} \label{hermite} For all $n$, we have the following representation of Hermite polynomials :  \[ H_n(x)=  \frac{(-1)^n n!}{2^n} \, e^{4 x^2} \sum _{k=\lceil n/2\rceil}^{\infty }
    \enom{k}{n}{2}  \frac{(-4)^k x^{2 k-n}}{k!}. \]\end{corollary}

\subsection{On the zeros of $P_n^{(m)}$ - A conjecture} The rational form of $\mathscr{F}_n(u)$ in Proposition~\ref{cgf} is characteristic of the generating functions of polynomials. The polynomials $P_n^{(m)}$ play the role of Eulerian polynomials appearing in the numerator of the generating function $\sum_k k^n u^k = A_n(u)/(1-u)^{n+1}$~\cite[p.209]{Stanly}. The Eulerian polynomials $A_n(u)$ are known to have all zeros real~\cite{harper}. It is quite normal to see if this property is also valid for the polynomials $P_n^{(m)}$.

If we take a look at the polynomial~\eqref{pm2}, we observe that a non-trivial zero of it (i,e. $\neq 0$) must be such that $(u-\sqrt{u(4-3 u)
   })/(u+\sqrt{u(4-3 u)})$ is an $(n+1)$-th root of unity. If $u_p$ denote such zeros then ( $i=\sqrt{-1}$ ) \[u_p=\frac{\left(1+e^{i \frac{2 \pi  p}{n+1}}\right)^2}{e^{i\frac{4 \pi  p}{n+1}}+e^{i \frac{2 \pi  p}{n+1}}+1} =2 \,\frac{1+\cos \left(\frac{2 \pi  p}{n+1}\right)}{1+2 \cos \left(\frac{2 \pi  p}{n+1}\right)},\] with $p \not\in\{(n+1)/3, \, 2(n+1)/3\}$ whenever $n+1$ is a multiple of 3. Thus the polynomials~\eqref{pm2} has real zeros only. Similar investigations for the case $m=3$ leads to the same conclusion. Actually, several numerical experimentations suggest forcibly the truth of
\begin{conjecture}For all $m \geq  1$, the colorless polynomials $P_n^{(m)}$ have real zeros only.
\end{conjecture}

\section{Asymptotics : the entropy density function} \label{s6}
This section is devoted to the study of a function that
characterizes the polynomial triangles in the limit where the row
index $k$ tends to infinity and the column index $n$ increases
proportionally, namely the asymptotic \emph{entropy density
function}. But before defining this notion which originates from
statistical mechanics and information theory, we recall a general formula that we shall
rely upon :

\begin{thm} \emph{(Daniels~\cite[p.646]{daniels}, Good~\cite[p.868]{good})} For a power series or a polynomial $f(t)$ with
non-negative real coefficients and a strictly positive radius of
convergence, define
\[\Delta f(t) = t\frac{f'(t)}{f(t)}\quad ; \quad \delta \!f(t) = \frac{f''(t)}{f(t)}-
\left(\frac{f'(t)}{f(t)}\right)^2 +  \frac{f'(t)}{t f(t)}.\] Assume that the function $f(t)$ is \emph{aperiodic}, i.e,
$\gcd \{i,\, [t^i]f(t) > 0\}=1$, and suppose that the equation $\Delta f(t) = n/k$ has a real positive
solution $x$ smaller than the radius of convergence of $f$. Then,
for $n,k \rightarrow +\infty$ and $n/k$ finite, \be \label{daniels} \left[t^n
\right]\left(f(t)\right)^k = \frac{f^k(x)}{x^{n+1}\sqrt{2 \pi k \;
\delta \!f(x)}} \left(1 + o(1)\right),\ee uniformly as $k
\rightarrow \infty$. \hfill{$\Box$}
\end{thm}
The ratio $n/k$, denote it $\rho$, is the mean number of balls in
one box. As a function of the saddle point $x$, $\rho$ is strictly
increasing given that $x \partial_x \rho(x) = x^2 \delta f(x)$ is
the variance of the (non-degenerate) random variable taking a value $i \in \{0, \ldots ,m\}$ with probability
$a_i x^i /f(x)$. The function $\rho(x)=\Delta f(x)$ itself being the
expectation value of this variable.

\begin{remark}The Daniels-Good theorem leads to the known asymptotic of the central trinomial coefficient, i.e, $n/k=1$ (A002426) : For $\a=(a_0,a_1,a_2)$; $a_i >0$, a little calculation gives
\[\pnom{k}{k}{\a}\sim \frac{(a_1+2\sqrt{a_0a_2})^{k+1/2}}{2 \sqrt[4]{a_0 a_2}\sqrt{\pi k}} \quad \hbox{as} \quad  k \to \infty.\]
Choosing $\a=(1,2,1)$, we recover also the asymptotic of central binomial coefficient
${2k \choose k} \sim 4^k/\sqrt{\pi k}$ as $k \to \infty$. \end{remark}
\subsection{Entropy density function} We define the \emph{entropy density function} as follows \begin{definition}When $k$ goes to infinity
and $\rho$ is fixed (the so-called thermodynamical limit), we define the
entropy density function (or entropy per box) as the limit \[
\lim_{k \rightarrow \infty} \frac{1}{k} \ln \pnom{k}{\rho k}{\a}
\stackrel{\mbox{\tiny def}}= h^{(\a)}(\rho)\,, \qquad 0 \leq \rho \leq m\] The existence of the limit is guaranteed by the Daniels-Good theorem.\end{definition}

In what follows, we assume that the $a_i$'s  are non-negative and the polynomial $p_\a$ is aperiodic
(in other words, there exists no integer $r$ such that
$p_{\a}(t)=\sum_i a_{r i}t^{ri}$). Thus, specializing~\eqref{daniels} for the polynomial coefficients, we get for a given $\rho \in [0,m]$ \be
\label{entropy} h^{(\a)}(\rho)=\ln p_{\a}(x(\rho)) - \rho \ln
x(\rho) , \ee where $x(\rho)$ is the
(unique) real positive zero of the polynomial $\sum_{i=0}^m
(i-\rho) a_i t^i$.  Explicit expressions of the entropy density function can be found for $m\leq 4$. For
instance
\[h^{((a_0, \,a_1))}(\rho)= (1-\rho)\ln a_0 + \rho \ln a_1 - \rho \ln \rho - (1-\rho) \ln(1-\rho),\]
which coincide, in the uncolored case, with the entropy function for the Bernoulli trial with parameter $\rho$ as probability of success.


\begin{figure}\begin{center}
\includegraphics[width=10cm]{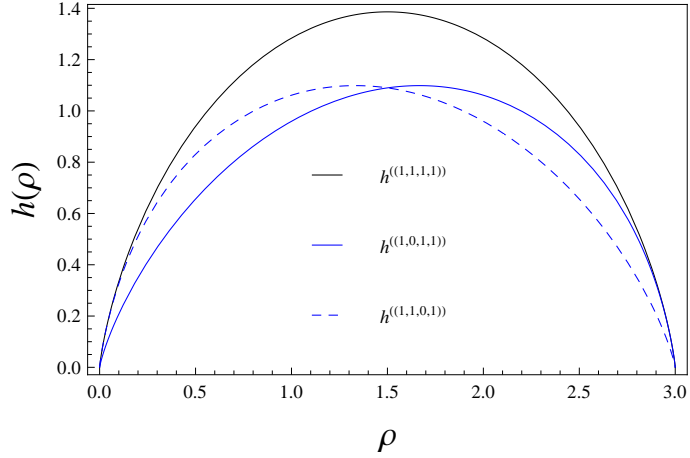}
  \caption{Entropy density function vs $\rho=n/k$ for the un-weighted
quadrinomial coefficients. }\end{center}
\end{figure}

Obviously, the function~\eqref{entropy} is continuous and
differentiable for $0 <\rho < m$. We now prove the main claim of this section.
\begin{theorem} \label{ent} The density function $h^{(\a)}$ fulfils the following properties
\begin{enumerate}
    \item[\emph{(i)}] $h^{(\a)}$ is strictly concave;
    \item[\emph{(ii)}] $h^{(\a)}$ is unimodal and reaches its peak at the point $\mu=\sum_i i a_i/\sum_i a_i$
    and \be \max_{0 \leq \rho \leq m} h^{(\a)}(\rho)=\ln\left(\sum_{i=0}^m a_i \right) ;\ee
    \item[\emph{(iii)}] $h^{(\a)}(\rho) \geq 0$ for all $\rho$, whenever $a_0 \geq 1$ and $a_m \geq 1$;
    \item[\emph{(iv)}] Particle-Hole duality \be h^{(\a)}(\rho)=h^{(\mathrm{J} \a)}(m-\rho) ; \ee
    \item[\emph{(v)}] As a function of $\a$: \be \Theta_\a h^{(\a)}(\rho)=1, \quad  \forall \rho\ee where $\Theta_\a=\sum_{i=0}^m a_i \partial_{a_i}$ is the Theta (or homogeneity) operator.
\end{enumerate}
\end{theorem}
\begin{proof}{~}

\noindent (\emph{i})  \; Consider the random variable $\xi$ whose probability mass function is given by $\mathbb{P}(\xi=i)=a_i x^i/ p_\a(x)$, for $i= 0, \ldots, m$. As noted above, the variance of $\xi$ is given by \[\mathbb{V}(\xi)=x \partial_x \rho (x) = x^2 \delta p_\a(x).\] Moreover, from~\eqref{entropy}, we derive \be\label{deriv} \partial_\rho h^{(\a)} (x)=-\ln x(\rho).\ee Whence \[\partial_\rho^2 h^{(\a)} (x)=- \frac{1}{x(\rho)}\partial_\rho x(\rho)=-{1 \over  \mathbb{V}(\xi)} < 0.\] Thus,  the entropy density function is strictly concave.
\medskip

\noindent {(\emph{ii})} \;Since $h^{(\a)}$ is concave, its maximum is attained when $x(\rho)=1$ as we can read from~\eqref{deriv}, i.e., when $h^{(\a)}(\rho(x=1))=\ln p_\a(1)$ according to~\eqref{entropy}. To show that  $h^{(\a)}$ is monotonically increasing for $0 < \rho < \mu$ and monotonically decreasing for $\mu < \rho < m$, recall that \be \label{rho}\rho(x)= x \frac{p_\a'(t)|_{t \rightarrow x}}{p_\a(x)},\ee and remark that $\partial_\rho h^{(\a)}(x)> 0$ if \, $0 < x < 1$, i.e., $0=\rho(0) < \rho(x) < \rho(1)= \mu$, because the function $\rho(x)$ is strictly increasing as noted above, and $\partial_\rho h^{(\a)}(x) <  0$ if \, $ x>1$, i.e., $\rho(x) > \mu$.
\medskip

\noindent (\emph{iii}) \;We see from~\eqref{rho} that if $\rho$ goes to 0, then $x \rightarrow 0$, since the zeros of $p_\a'(t)|_{t \rightarrow x}$ are essentially negative and if $\displaystyle \rho \rightarrow m^+$ then $x \rightarrow +\infty$. Moreover, from~\eqref{entropy} we derive that
\[\lim_{\rho \rightarrow 0} h^{(\a)}(\rho)=\lim_{x \rightarrow 0} h^{(\a)}(\rho(x))= \ln a_0 ,\] and \[\lim_{\rho \rightarrow m} h^{(\a)}(\rho)= \lim_{x \rightarrow +\infty} h^{(\a)}(\rho(x))=\ln a_m.\] These limits together with the strict concavity implies  that the entropy density is non-negative if $a_0 \geq 1$ and $a_m \geq 1$ .
\medskip

\noindent (\emph{iv}) The Particle-Hole duality is obvious from the polynomial symmetry (Table~\ref{tab1}).
\medskip

\noindent (\emph{v})\;  In fact  the differential equation is valid for all $k$ and $n$, in particular in the thermodynamical limit: \[ \Theta_\a \frac{1}{k}\ln \pnom{k}{n}{\a}=\frac{1}{k} \pnom{k}{n}{\a}^{-1} \Theta_\a \pnom{k}{n}{\a}=1,\] where we have used the homogeneity of~\eqref{pnom} as a polynomial in $\a$:  $\Theta_\a \pnom{k}{n}{\a}=k \pnom{k}{n}{\a}$ which follows readily from the trivial formula $\Theta_\a p_\a^k(t)=k p_\a^k(t)$. \end{proof}


\section{Concluding remarks}\label{s7} In his paper, Richard C. Bollinger~\cite{Boll3}, concludes with the hopes that, \begin{center} \begin{minipage}[b]{4 in}``like $\mathsf{T}_1$ has certainly been a rich source of interesting and useful mathematics, its extended relatives (i.e, $\T$) potentially may serve as equally fruitful objects of study.'' \end{minipage}\end{center} Our study concretizes in some extent the Bollinger's suggestion. We too believe that deeper aspects still be discovered for the polynomial triangles. The following items give a sample of our observations that make interesting exercises :
\begin{enumerate}[1.]
\item Through the expression~\eqref{gegen}, several recurrences of Gegenbauer polynomials can be rewritten in terms of trinomial coefficients and, possibly, could be extended to general polynomial ones. For instance, we find~\cite{wolfram} \begin{eqnarray} \nonumber (2 k-n-1) (2 k-n)\enom{k}{n}{2}&=&k (7 k-3 n-5) \pnom{k-1}{n}{2}-3 (k-1) k \pnom{k-2}{n}{2}\\
\nonumber 2(k+1) \pnom{k}{n}{2}&=&(2 k-n+2)\pnom{k+1}{n}{2}
   -(k+1)\pnom{k}{n-1}{2}\\
\nonumber  n (2 k-n)\pnom{k}{n}{2} &=&k(2 k-1)
   \pnom{k-1}{n-1}{2}+ 3 (k-1)k \pnom{k-2}{n-2}{2}.\end{eqnarray} Can we find recurrences of the same type for the coefficients~\eqref{pnom}?
  \item The exponential generating function of row sequences of $\mathsf{T}(\a)$ : \[\sum_{n} \pnom{k}{n}{\a}\frac{t^n}{n!}\]provides a natural extension of Laguerre polynomials : $L_k(t)=\sum_n \binom{k}{n}(-t)^n/n!$. Can orthogonality and other properties of Laguerre polynomials be generalized?
   \item It may also be of interest to extend the present work to multivariate polynomials.
\end{enumerate}

\section{ Acknowledgments} The author is grateful to Professor El Hassane Saidi, director of Lab/UFR- High Energy Physics of Mohammed V University, for his kind support and generous encouragements. He thanks Adil Belhaj and Mohammed Daoud for useful discussions.

\end{document}